\documentclass[11pt,oneside]{amsart}

\usepackage{amsmath,ifthen, amsfonts, amssymb,
srcltx, 
   amsopn,tikz}
\usetikzlibrary{decorations.markings,arrows}

\usepackage{graphicx}

\newcommand{\showcomments}{yes}
\renewcommand{\showcomments}{no}

\newsavebox{\commentbox}
\newenvironment{com}%
{\ifthenelse{\equal{\showcomments}{yes}}%
{\footnotemark
        \begin{lrbox}{\commentbox}
        \begin{minipage}[t]{1.25in}\raggedright\sffamily\tiny
        \footnotemark[\arabic{footnote}]}
{\begin{lrbox}{\commentbox}}}%
{\ifthenelse{\equal{\showcomments}{yes}}%
{\end{minipage}\end{lrbox}\marginpar{\usebox{\commentbox}}}
{\end{lrbox}}}

\newtheorem{thm}{Theorem}[section]
\newtheorem{lem}[thm]{Lemma}

\newtheorem{cor}[thm]{Corollary}
\newtheorem{conj}[thm]{Conjecture}
\newtheorem{prop}[thm]{Proposition}

\theoremstyle{definition}
\newtheorem{defn}[thm]{Definition}

\newtheorem{prob}[thm]{Problem}

\DeclareMathOperator{\rank}{rank}
\DeclareMathOperator{\degree}{degree}

\DeclareMathOperator{\link}{link}

\newcommand{\field}[1]{\mathbb{#1}}
\newcommand{\integers}{\ensuremath{\field{Z}}}

\newcommand{\reals}{\ensuremath{\field{R}}}

\newcommand{\alink}{\link_{\uparrow}(x)}
\newcommand{\dlink}{\link_{\downarrow}(x)}

\newcommand{\euler}{\chi}

\newcommand{\pr}[1]{\mathbb P( #1)}
\newcommand{\p}{\textup{\textsf{p}}}

\newcommand{\mcl}{\mathcal}

\DeclareMathOperator{\deficiency}{{def}}

\begin{document}

\title{Incoherent Coxeter Groups}

\author{Kasia Jankiewicz}
\author[D.~T.~Wise]{Daniel T. Wise}
	\address{Dept. of Math. \& Stats.\\
                    McGill University \\
                    Montreal, Quebec, Canada H3A 0B9}
           \email{kasia@math.mcgill.ca}
           \email{wise@math.mcgill.ca}

\date{\today}
\thanks{Research supported by NSERC}

\begin{com}
{\bf \normalsize COMMENTS\\}
ARE\\
SHOWING!\\
\end{com}

\begin{abstract}
We use probabilistic methods to prove that many Coxeter groups are incoherent. In particular, this holds for Coxeter groups of uniform exponent $>2$ with sufficiently many generators. \end{abstract}

\maketitle

\begin{com}
TO DO:
 Intro?
\end{com}
\section{Introduction}
A \emph{Coxeter group} $G$ is given by the presentation
\[
\langle a_1,\dots, a_r\mid a_i^2, (a_ia_j)^{m_{ij}} : 1\leq i < j\leq r \rangle
\]
where $m_{ij}\in\{2,3,\dots,\infty\}$ and where $m_{ij}=\infty$ means no relator of the form $(a_ia_j)^{m_{ij}}$. Throughout this paper all presentations of Coxeter groups are of the above form. It is traditional to encode the above data for $G$ in terms of an associated labelled graph $\Upsilon_G$, whose vertices correspond to the generators and where an edge labelled by $m_{ij}$ joins vertices $a_i, a_j$ when $m_{ij}<\infty$. We omit an edge for $m_{ij}=\infty$.
\begin{defn}
A group $G$ is \emph{coherent} if every finitely generated subgroup of $G$ is finitely presented. Otherwise, $G$ is \emph{incoherent}.
\end{defn}


Our main result which is stated and proven as Theorem~\ref{thm:nonuniform case} is the following:
\begin{thm}
For each $M$ there exists $R=R(M)$ such that if $K$ is a Coxeter group with $3\leq m_{ij}\leq M$ and rank $r\geq R$ then $K$ is incoherent.
\end{thm}

Our result joins a similar result for groups acting properly and cocompactly on Bourdon buildings \cite{WiseRandomMorse} and we expect that there is more to come in this direction.

\section{Preliminaries on Coxeter groups, Walls, and Morse Theory}

\subsection{Euler characteristic and compression}\label{sec:euler}
Let $G$ be a Coxeter group given by
\[
\langle a_1,\dots, a_r\mid a_i^2, (a_ia_j)^{m_{ij}} : 1\leq i < j\leq r \rangle
\]
and let $X$ be the standard $2$-complex associated to this presentation.
Consider an index $d$ torsion-free subgroup $G'$ of $G$. Let $\widehat X\to X$ be a cover of $X$ corresponding to $G'$. All edges embed in $\widehat X$, since all generators are torsion elements, and all $2$-cells embed since each proper subword of $(a_ia_j)^{m_{ij}}$ is a torsion element. Consider the complex $\overline{X}$ obtained from $\widehat X$ by firstly collapsing $2$-cells corresponding to $a_i^2$ relators to $1$-cells and secondly collapsing $2m_{ij}$ copies of $2m_{ij}$-gons with the same boundary when $m_{ij} \neq \infty$.
The complex $\overline X$ is the \emph{compression} of $\widehat X$. See Figure~\ref{fig:compression} for the compression arising from $\langle a,b\mid a^2, b^2, (ab)^3\rangle$.

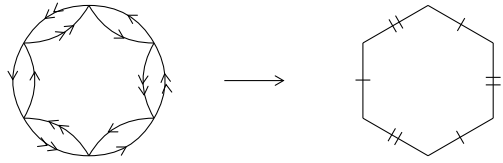
\begin{figure}
\begin{tikzpicture}[scale=1]
\tikzset{%
    > = angle 60,
    arrow/.style = {%
      decoration = {%
        markings,
        mark = at position .5 with {\arrow {>};}
      },
      postaction = decorate
    }
   }

\tikzset{%
    > = angle 60,
    arrowback/.style = {%
      decoration = {%
        markings,
        mark = at position .5 with {\arrow {<};}
      },
      postaction = decorate
    }
  }
  \tikzset{%
    > = angle 60,
    darrow/.style = {%
      decoration = {%
        markings,
        mark = at position .6 with {\arrow {>>};}
      },
      postaction = decorate
    }
  }
 \tikzset{%
    > = angle 60,
    darrowback/.style = {%
      decoration = {%
        markings,
        mark = at position .57 with {\arrow {<};}
      },
       decoration = {%
        markings,
        mark = at position .43 with {\arrow {<};}
      },
      postaction = decorate
    }
  }

\draw[arrowback] (30:1) to[in=-60, out=180] (90:1);
\draw[darrowback](90:1) to[out=240,in=0] (150:1);
\draw[arrowback](150:1) to[out=-60, in=60] (210:1);
\draw[darrowback](210:1) to[out=0, in=120] (270:1);
\draw[arrowback](270:1) to[out=60, in=180] (330:1);
\draw[darrowback](330:1) to[out=120, in=-120] (30:1);

\draw[arrow] (30:1) to[in=0, out=120] (90:1);
\draw[darrow](90:1) to[out=180,in=60] (150:1);
\draw[arrow](150:1) to[out=-120, in=120] (210:1);
\draw[darrow](210:1) to[out=-60, in=180] (270:1);
\draw[arrow](270:1) to[out=0, in=-120] (330:1);
\draw[darrow](330:1) to[out=60, in=-60] (30:1);

\draw[->] (1.8,0) to (2.6,0);
\end{tikzpicture}\,\,\,\,\,\,\,\,\,\,\,\,\,\,
\begin{tikzpicture}[scale=1]

  \tikzset{%
    > = stealth,
    notarrow/.style = {%
      decoration = {%
        markings,
        mark = at position .5 with {\arrow{|};}
      },
      postaction = decorate
    }
  } 
    \tikzset{%
    > = stealth,
    dnotarrow/.style = {%
      decoration = {%
        markings,
        mark = at position .45 with {\arrow {|};}
        },
       decoration = {%
       	markings,
	mark = at position .55 with {\arrow{|};}
      },
      postaction = decorate
    }
  } 
\draw[notarrow] (30:1) to (90:1);
\draw[dnotarrow] (90:1) to (150:1);
\draw[notarrow] (150:1) to (210:1);
\draw[dnotarrow] (210:1) to (270:1);
\draw[notarrow] (270:1) to (330:1);
\draw[dnotarrow] (330:1) to (30:1);
\end{tikzpicture}
\caption{The compression $\widehat X\to\overline X$. Each bigon collapses to an edge and six $2$-cells collapse to a single $2$-cell.}\label{fig:compression}
\end{figure}

We say $G$ has \emph{dimension~$\leq 2$} when $\overline X$ is aspherical. This holds precisely when $\frac{1}{m_{ij}}+\frac{1}{m_{jk}} +\frac{1}{m_{ki}} \leq 1$ for each $i,j,k$. Indeed, there is then  a natural metric of nonpositive curvature on $\overline X$ induced by metrizing each $2$-cell as a regular Euclidean polygon. However if some 3-generator subgroup is finite, then $\overline X$ contains a copy of $S^2$. We focus on Coxeter groups of dimension~$\leq 2$, in which case the following discussion of $\euler(G)$ is sensible.

The complex $X$ has one $0$-cell, $r$ $1$-cells and one $2$-cell for each pair of generators $\{i,j\}$ with $m_{ij}< \infty$. As $\degree(\widehat X\rightarrow X)=d$, the complex $\widehat X$ has $d$ $0$-cells, $dr$ $1$-cells and $d$ $2$-cells for each pair $\{i,j\}$ with $m_{ij}<\infty$. The complex $\overline X$ has $d$ $0$-cells, $\frac{dr}{2}$ $1$-cells and $\frac{d}{2m_{ij}}$ $2$-cells for each pair $\{i, j\}$ with $m_{ij}<\infty$. The \emph{Euler characteristic of $G$} is:
\begin{equation}\label{eulerformula}
\euler(G) = \frac{\euler(\overline X)}{[G:G']} = \frac{1}{d}\bigg(d - \frac{dr}{2} +\sum_{\{i,j\}}\frac{d}{2m_{ij}}\bigg) = 1 - \frac{r}{2} +\sum_{\{i,j\}}\frac{1}{2m_{ij}}.
\end{equation}
This is independent of the choice of finite index torison-free subgroup.
We thus have:
\[\euler(G_{(r,m)}) = 1 - \frac{r}{2} + \frac{(r-1)r}{4m}.\]
Thus if $m$ is fixed then $\euler(G_{(r,m)})>0$ for all sufficiently large $r$.

\subsection{Walls}
Let $K$ be a combinatorial $2$-complex
with the property that each $2$-cell has an even number of sides (we have in mind $K = \overline X$ as defined in previous section).
Two $1$-cells in the attaching map $\partial_\p C\to K^{1}$
of a $2$-cell $C$ are \emph{parallel}
if they are images of opposite edges in $\partial_\p C$.
An \emph{abstract wall} is an equivalence class
of $1$-cells in the equivalence relation generated by parallelism. A \emph{wall} $W$ associated to an abstract wall $\bar W$ is a graph with a locally injective map $\phi:W\to K$ defined as follows:
\begin{itemize}
\item for each $1$-cell $a$ in $\bar W$ there is a vertex $v_a$ in $W$,
\item $\phi(v_a)$ is the center of $a$,
\item for each pair of $1$-cells $a,a'$ in $\bar W$ and each $2$-cell $C$ in which $a,a'$ are parallel, there is an edge $(v_a,v_{a'})$ in $W$,
\item the edge $(v_{a},v_{a'})$ is sent by $\phi$ to an arc in $C$ joining $\phi(v_a)$ and $\phi(v_{a'})$.
\end{itemize}
The wall $W$ is \emph{dual} to each $1$-cell in $\bar W$. The wall $W$ is \emph{adjacent to $x$ at a vertex $v$} of $\link(x)$, if it is dual to the $1$-cell corresponding to $v$. The wall $W$ is \emph{adjacent to $x$ at an edge $e$} of $\link(x)$, if $W$ is not adjacent at either endpoint of $e$ but is dual to a pair of $1$-cells in $\partial_\p C$ where $C$ is a $2$-cell corresponding to $e$. We say that the wall $W$
\begin{itemize}
\item \emph{embeds} if $W\to K$ is injective,
\item is \emph{two-sided} if $W\to K$ extends to an embedding $W\times (-1,1)\to K$,
\item \emph{self-osculates at $x$} if it is adjacent to $x$ at more than one vertex and/or edge of $\link(x)$. See Figure \ref{self-osculation}.
\end{itemize}
\begin{figure}
\begin{tikzpicture}[scale=0.5]
\draw[thick] (210: 1) to (240: 1.732) to (270: 2) to (300:1.732) to (330: 1) to (0:1.732) to (30:2) to (60:1.732) to (90:1) to (120:1.732) to (150:2) to (180:1.732) to (210:1) to (0:0) to (90:1);
\draw[thick]  (330:1) to (0:0);
\path[draw,thick,red] (-2.2,0.65) to[out=-30,in=180] (-1.75, 0.5) to[out=0,in=180] (1.75,0.5) to[out = 0, in = 120] (3,0) to[out=300,in=10] (2.5, -1.75) to[out=190,in=0] (0.5*1.732,-1) to[out =180, in=0] (-0.5*1.732, -1) to[out=180,in=60] (-1.5,-1.2);
\end{tikzpicture}
\begin{tikzpicture}[scale = 0.5]
\draw[thick] (-1,0) to (-1.6, 0.6) to (-1.6, 1.6) to (-0.6, 1.6) to (0, 1) to (0.6, 1.6) to (1.6,1.6) to (1.6, 0.6) to (1,0) to (1.6, -0.6) to (1.6, -1.6) to (0.6, -1.6) to (0, -1) to (-0.6, -1.6) to (-1.6, -1.6) to (-1.6, -0.6) to (-1,0) to (0,0) to (1,0);
\draw[thick] (0,1) to (0,-1);
\path[draw, thick, red] (-2.1, 1.3) to[in=180, out=-40] (-1.6, 1.1) to[in=180, out=0] (0,1/2) to[in=180, out=0] (1.6, 1.1) to[in=90, out=0] (2.8,0) to[in=0, out=270] (1.6, -1.1) to[in=0, out=180] (0,-1/2) to[in=0, out = 180] (-1.6, -1.1) to[in=40, out=0] (-2.1, -1.3);
\end{tikzpicture}\,\,\,\,
\begin{tikzpicture}[scale = 0.5]
\draw[thick] (210: 1) to (240: 1.732) to (270: 2) to (300:1.732) to (330: 1) to (0:1.732) to (30:2) to (60:1.732) to (90:1) to (120:1.732) to (150:2) to (180:1.732) to (210:1) to (0:0) to (90:1);
\draw[thick]  (330:1) to (0:0);
\path[draw, thick ,red] (-3*1.732/4 - 0.5, -0.45) to [in=240, out= 10] (-3*1.732/4,-0.25) to[in=240,out=60] (-1.732/4,1.25) to[in=170,out=60] (1.5,2) to[in=110,out=-10] (2.8,1) to[in=10,out=290] (2,-0.9) to[in=0,out=190] (0.866, -1) to[in=0,out=180] (-0.866,-1) to[out=180, in=60] (-1.2, -1.2);
\end{tikzpicture}

\caption{In each case above, the wall self-osculates at the central vertex.}\label{self-osculation}
\end{figure}
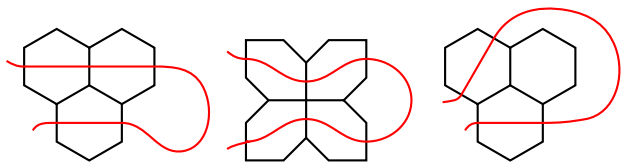

\subsection{Orientation of walls}
An embedded wall $W\to K$ is two-sided if and only if there is a globally consistent orientation of its dual $1$-cells such that parallel $1$-cells in any $2$-cell $C$ have opposite orientations in $\partial_\p C$. An \emph{orientation} of a two-sided wall $W$ is one of two globally consistent orientations of its dual $1$-cells. Let $\mathcal W$ be the set of all walls in $K$. An \emph{orientation} on $\mathcal W$ is a choice of orientation on each $W\in\mathcal W$.

\subsection{Bestvina-Brady Morse theory}
An \emph{affine complex} $K$ has cells that are convex Euclidean polyhedra, which metrically agree on their faces.
A map $f:K \to \reals$ is a \emph{Morse function} if it is linear on each cell $C$, constant on $C$ if only if $\dim C = 0$, and the image $f(K^0)$ of the $0$-skeleton is a closed discrete subset of $\reals$. It follows that the restriction of $f$ to a cell has a unique minimum and maximum.

Let $x\in K^0$. A vertex $v\in \link(x)$ is \emph{ascending} (resp. \emph{descending}) if the corresponding $1$-cell is oriented away from $x$ (resp. toward $x$). An edge $e\in\link(x)$ is \emph{ascending} (resp. \emph{descending}) if each wall passing through the corresponding $2$-cell is oriented away from $x$ (resp. toward $x$). The \emph{ascending link} $\link_{\uparrow}(x)$ (resp. \emph{descending link} $\link_{\downarrow}(x)$) is the subgraph of $\link(x)$ consisting of all ascending (resp. descending) vertices and edges.

We will employ the following result of Bestvina-Brady proven in \cite[Thm~4.1]{BestvinaBrady97}:
\begin{thm} \label{bestvina-brady thm} Let $K$ be a finite (aspherical) affine cell complex. Consider a map $K\to S^1$ that lifts to a Morse function $\widetilde K\to \mathbb R$. If $\link_{\uparrow}(x)$ and $\link_{\downarrow}(x)$ are nonempty and connected for each $x\in \widetilde K^{0}$, then $\ker(\pi_1K\to \mathbb Z)$ is finitely generated.\end{thm}

\subsection{An orientation induces a combinatorial map $K^{1}\to S^1$}\label{orienatation induces a map} Let $S^1$ have a cell structure with one $0$-cell and one (oriented) $1$-cell. Each orientation on $\mathcal W$ determines an orientation preserving combinatorial map $K^{1}\to S^1$. The map $\partial_\p C\to S^1$ is null-homotopic for each $2$-cell $C$, since pairs of opposite $1$-cells in $\partial_\p C$ travel in opposite directions around $S^1$. Thus the map $K^{1}\to S^1$ extends to $K \to S^1$. The map $K\to S^1$ lifts to $\widetilde K\to \widetilde S^1\simeq \mathbb R$, but the restriction of this map to a $2$-cell does not necessarily have a unique minimum or maximum.

The \emph{lawful subcomplex}
 $Y\subset K$ is the subcomplex of $K$ obtained by discarding $2$-cells whose attaching maps cannot be expressed as the concatenation $\alpha\beta^{-1}$ where $\alpha\to K^1$, $\beta\to K^1$ are positively directed paths. The restriction $Y\to S^1$ of the map $K\to S^1$ lifts to $\widetilde Y \to \mathbb R$ which is a Morse function in the sense of Bestvina-Brady.

\section{Main Theorem}
The \emph{Coxeter group of uniform exponent $m$ and rank $r$} is the Coxeter group $G_{(r,m)}$ with the following presentation:
$$\langle a_1,\ldots, a_r \mid a_i^2, (a_ia_j)^m \ : \ 1 \leq i <  j \leq r\rangle$$
The standard $2$-complex of the above presentation for $G_{(r,m)}$ is denoted by $X_{(r,m)}$.

\begin{thm}\label{thm:uniform case}
For each $m\geq 3$ there exists $R_m$ such that for all $r\geq R_m$
the group $G_{(r,m)}$ has a finite index torsion-free subgroup $G'$ that admits an epimorphism $G'\rightarrow \integers$ whose kernel $N$ is finitely generated.
\end{thm}

\begin{cor}\label{cor:uniform}For $m\geq 3$, the group
$G_{(r,m)}$ is incoherent for all sufficiently large $r$.
\end{cor}
\begin{proof}A result of Bieri in \cite{BieriBook81} states that a nontrivial finitely presented normal subgroup of a group of cohomological dimension $\leq 2$ is either free or of finite index.
Since $[G':N] = \infty$ it remains to exclude the case where $N$ is free, whence:
\[
\euler(G') = \euler(N)\cdot \euler(\integers) = (1-\rank(N))\cdot 0 = 0.\]
This is impossible for all sufficiently large $r$, since then $\euler(G')>0$ (see Section~\ref{sec:euler}).
\end{proof}

A {\em Coxeter subgroup} is generated by a subset of the generators of $G$.
It is presented by those generators together with all relators
in those generators appearing in the presentation of $G$~\cite{DavisCoxeterBook2008}. We now prove the main result stated in the introduction:
\begin{thm}\label{thm:nonuniform case}
For each $M$ there exists $R=R(M)$ such that if $K$ is a Coxeter group with $3\leq m_{ij}\leq M$ and rank $r\geq R$ then $K$ is incoherent.
\end{thm}
\begin{proof}
The multi-color version of Ramsey's theorem \cite{GrahamRothschildSpencer80} states that given a number of colors $c$ and natural numbers $n_1,\dots, n_c$ there exists a number $R = R(n_1,\dots, n_c)$ such that if the edges of a complete graph $\Gamma$ of order at least $R$ are colored with $c$ colors, then for some $i$ there exists a complete subgraph of $\Gamma$ of order $n_i$ with edges of color $i$. Let $c = M$ and $n_i = R_{i}$ of Theorem~\ref{thm:uniform case}. Consequently there exists a uniform exponent Coxeter subgroup $G_{(r,m)}$ of $K$ for some $m\leq M$ and $r = R_m$. By Corollary~\ref{cor:uniform} the subgroup $G_{(r,m)}$ is incoherent and hence so is $K$.
\end{proof}

The above results lend credence to the following:
\begin{conj}\label{conj:positive incoherent}
Let $G$ be a finitely generated infinite Coxeter group of dimension~$\leq 2$. If $\euler(G)>0$ then $G$ is incoherent.
\end{conj}

\subsection{A polynomial degree finite cover of $X_{(r,m)}$ with good walls}\label{finite cover r}
The goal of this subsection is to prove the following:
\begin{prop}\label{prop:finite cover}
There is a homomorphism $\beta:G_{(r,m)}\to Q^{k(r)}$ such that the compression $\overline X_{(r,m)}$ of the induced cover $\widehat X_{(r,m)}\to X_{(r,m)}$ has the following property: each wall is $2$-sided, embedded and has no self-osculation.

Moreover $|\overline X^0_{(r,m)}|$ is at most $|Q|^{k(r)}\leq |Q| r^C$ for some constant $C$.
\end{prop}

The proof of Proposition~\ref{prop:finite cover} appears at the end of this subsection.

A \emph{partition} of a set $S$ is a map $p:S\to \{1,2,3,4\}$. The partition $p$ \emph{separates}
  $a,b,c,d$  if $p(a), p(b), p(c), p(d)$ are distinct.
\begin{lem}\label{partitions}Let $S$ have cardinality $ r\geq 4$. There is a collection of $k =k(r)=\bigg\lceil \frac{\log{r\choose4}}{\log\frac {32} {29}} \bigg\rceil$ partitions such that each quadruple of distinct elements of $S$ is separated by this collection.
\end{lem}
\begin{proof} Let $M$ denote the set of all partitions of $S$, and note that $|M| = 4^r$.
Let $\mathcal M_k$ denote the collection of cardinality $k$ subsets of $M$ and note that $|\mathcal M_k|={4^r \choose k}$.
Let $\mathcal N_k\subset \mathcal M_k$ be the subcollection consisting of sets of $k$ partitions that do not separate some quadruple. We want to show that $|\mathcal N_k|<|\mathcal M_k|$. Let $\mathcal N_k(\{a,b,c,d\})\subset\mathcal M_k$ be the subcollection of sets that fail to separate $a,b,c,d\in S$. We have
\[
|\mathcal N_k|\leq {r\choose 4}|\mathcal N_k(\{a,b,c,d\})|
\]
since there are ${r\choose 4}$ quadruples $\{a,b,c,d\}$ of distinct elements of $S$.
There are $4! \cdot 4^{r-4} = 6\cdot 4^{r-3}$ partitions that separate $a,b,c,d$. Thus there are $4^r - 6\cdot 4^{r-3} = \frac{29}{32}\cdot 4^{r}$ partitions that do not separate $a,b,c,d$. We thus have
\[|\mathcal N_k(\{a,b,c,d\})|= {\frac{29}{32}\cdot 4^{r}\choose k}.\]
 Observe that we have the following:
\[
{\frac{29}{32}\cdot 4^{r}\choose k} < \Big(\frac{29}{32}\Big)^k{4^r\choose k}.
\]
Since $k\geq \log{r\choose 4}/\log\frac{32}{29}$ we have
\begin{align*}
{r\choose 4}\Big(\frac{29}{32}\Big)^k\leq1.\end{align*}
Altogether we have
\[
|\mathcal N_k|\leq {r\choose 4}|\mathcal N_k(\{a,b,c,d\})|
<{r\choose4}\Big(\frac{29}{32}\Big)^k{4^r\choose k}\leq {4^r\choose k}=|\mathcal M_k(P)|.
\qedhere\]
\end{proof}

\begin{proof}[Proof of Proposition~\ref{prop:finite cover}]
There is a finite quotient $\psi:G_{(4,m)}\twoheadrightarrow Q$ such that $\ker \psi$ is torsion-free, and the compression $\overline X_{(4,m)}$ of the induced cover $\widehat X_{(4,m)}\to X_{(4,m)}$ has the following property: each wall in $\overline X_{(4,m)}$ is $2$-sided, embedded and has no self-osculation. This follows from the separability of wall stabilizers  \cite{HaglundWiseCoxeter}.

Let $S=\{1,\dots,r\}$. Each partition $p:S\to \{1,2,3,4\}$ defines a homomorphism $\phi_p:G_{(r,m)} \to G_{(4,m)}$ induced by $\phi_p(a_i)=a_{p(i)}.$ Let $$\beta=(\psi\circ\phi_{p_1},\dots, \psi\circ\phi_{p_k}) : G_{(r,m)}\to Q^{k(r)}$$
 where $(p_1,\dots p_k)$ is a collection of partitions from Lemma~\ref{partitions}. For each partition $p$ there is a map $\overline{\phi}_p:X_{(r,m)}\to X_{(4,m)}$ induced by $\phi_p$. We will show that a ``wall pathology'' in $X_{(r,m)}$ would project to a wall pathology in $X_{(4,m)}$ for a suitable $p$ and hence there are no such wall pathologies. Suppose there is a wall $W$ in $X_{(r,m)}$ that self-intersects within a $2$-cell $C$. Let $a_i, a_j$ be the generators of $G_{(r,m)}$ labelling $C$. Let $p\in\{p_1,\dots, p_k\}$ separate $i$  and $j$. The image $\overline \phi_p(W)$ is a wall in $X_{(4,m)}$ that self-intersect, which is a contradiction. Thus walls in $X_{(r,m)}$ embed. We now show that no wall in $X_{(r,m)}$ has a self-osculation. Suppose $W$ in $X_{(r,m)}$ that has a self-osculation at some $0$-cell $x$, let $C,C'$ be $2$-cells adjacent to $x$ such that $W$ is dual to edges in both $C, C'$. Let $a_i, a_j, a_{i'}, a_{j'}$ be generators that label the boundaries of $C, C'$. If $i,j,i',j'$ are distinct consider a partition $p$ that separates them. The image $\overline \phi_p(W)$ is a wall in $X_{(4,m)}$ that has a self-osculation, which is a contradiction. Otherwise $C$ and $C'$ share one label and we let $p$ be a partition that separates the three distinct generators, and the argument is similar. Hence walls do not have self-osculations in $X_{(r,m)}$. Finally, the fact that all walls of $X_{(r,m)}$ are $2$-sided follows by considering a single $\overline\phi_p:X_{(r,m)}\to X_{(4,m)}$.

Finally, to see that the degree is bounded by a polynomial we observe that:
 \[
 |Q|^k \leq |Q|^{\frac{\log{r\choose 4}}{\log\frac {32} {29}}+1} = |Q| {r\choose 4}^{\frac{\log|Q|}{\log\frac {32} {29}}}\leq |Q|r^{\frac{4\log|Q|}{\log\frac {32} {29}}}. \qedhere
 \]
  \end{proof}

\subsection{Probability of empty or disconnected link is exponentially small}\label{probability}
Let $\Gamma$ be a complete graph on $r$ vertices.
%
%
Consider assigning a vertex to be ascending [respectively descending] with probability $\frac12$.
Furthermore, for an edge whose vertices are ascending [descending] assign it to be ascending [descending]
with probability $\frac{1}{2^{m-2}}$. Let $\Gamma_{\uparrow}$ [$\Gamma_{\downarrow}$] be the subgraph of $\Gamma$ consisting of all ascending [descending] vertices and edges.

Observe that  $\Gamma_{\uparrow}$ is assured to be nonempty and connected if
\begin{enumerate}
\item there exists an ascending vertex in $\Gamma$, and
\item for each pair of distinct vertices $v_1, v_2\in\Gamma_{\uparrow}$ there is a third vertex $v_3\in\Gamma_{\uparrow}$ such that $(v_1,  v_3)$ and $(v_2, v_3)$ are edges in $\Gamma_{\uparrow}$.
\end{enumerate}

Let $P_i$ denote the probability that condition~$(i)$ fails to be satisfied. If $\Gamma$ \emph{fails} to be nonempty and connected then at least one of $(1),(2)$ is not satisfied. Consequently
\[
\pr{\Gamma_{\uparrow}\text{ fails}}\leq P_1+P_2.\]
Similarly,
\[
\pr{\Gamma_{\downarrow}\text{ fails}}\leq P_1+P_2.\]

\begin{lem} $P_1 = \frac 1 {2^{r}}.$
\end{lem}
\begin{proof}Since no wall in $\mathcal W$ has a self-osculation, each wall is adjacent to $x$ at at most one vertex of $\Gamma$. Each of the $r$ vertices in $\Gamma$ is descending with probability $\frac 1 2 $ and these probabilities are independent. Hence $P_1=\frac 1{2^r}$.
\end{proof}

\begin{lem}
$P_2\leq {r\choose 2} \frac 1 4 (1-\frac 1{2^{2m-3}})^{r-2}.
$
\end{lem}

\begin{proof}
For distinct vertices $v_1,v_2\in\Gamma_{\uparrow}$ the edge  $(v_1,v_2)$ is ascending with probability $\frac 1 {2^{m-2}}$.
For a triple $v_1,v_2,v_3$ of distinct vertices in $\Gamma$, where $v_1, v_2$ are ascending the probability that $v_3$ is also ascending and both edges $(v_1, v_3), (v_2, v_3)$ are ascending is
\[\frac1 {2^{2m-3}}.\]
For $v_1,v_2\in\Gamma_{\uparrow}$ the probability that there is no connecting $v_3$ as above equals
\[(1-\frac 1{2^{2m-3}})^{r-2}.\]
Thus \[
P_2\leq\sum_{v_1, v_2\in \Gamma}\frac 1 4 (1-\frac 1{2^{2m-3}})^{r-2} = {r\choose 2} \frac 1 4 (1-\frac 1{2^{2m-3}})^{r-2}.\qedhere
\]
\end{proof}

Consider orientations on the set of all walls $\mcl W$ of $\overline X_{(r,m)}$. We orient each wall randomly, assigning probability $\frac 12$ to each of two orientations for each wall $W\in \mcl W$.
For each $0$-cell $x\in \overline X_{(r,m)}$ the graph $\link(x)$ is complete on $r$ vertices. No self-osculations in $\overline X_{(r,m)}$ provide that walls adjacent to two distinct edges and/or vertices of $\link(x)$ are distinct. Thus every vertex of $\link(x)$ is ascending [descending] with probability $\frac 12$ and each edge of $\link(x)$ whose edges are ascending [descending] is ascending [descending] with probability $\frac{1}{2^{m-2}}$. We thus have the following:
\begin{cor}\label{cor:probability}
$\pr{\alink\text{ or }\dlink\text{ fails}}$ is exponentially decreasing. Specifically
\begin{align*}
\pr{\alink\text{ or }\dlink\text{ fails}}
&\leq\; \pr{\alink\text{ fails}} + \pr{\dlink\text{ fails}}\\
&\leq\; 2(P_1+P_2)\leq \frac 1{2^{r-1}}+{r\choose 2} \frac 1 2 (1-\frac 1{2^{2m-3}})^{r-2}.
\end{align*}
\end{cor}

\subsection{Proof of Theorem~\ref{thm:uniform case}}
\begin{proof}
%
Proposition~\ref{prop:finite cover} provides a finite cover $\widehat X_{(r,m)}$ whose degree is bounded by a polynomial in $r$, and such that the  compression $\overline X=\overline X_{(r,m)}$
has the property that its walls are two-sided and have no self-osculations.

To apply Theorem~\ref{bestvina-brady thm} we need to find an orientation on $\mathcal W$ such that $\link_{\uparrow}(x)$ and $\link_{\downarrow}(x)$ are nonempty and connected for each $x\in\overline X^{0}$.  
%
We orient each $W\in\mcl W$ randomly assigning probability $\frac 12$ to each of two orientations of $W$. We need to prove that
\[
\pr{\link_{\uparrow}(x) \text{ or }\link_{\downarrow}(x) \text{ fails for some }x\in \overline X^{0}}<1.
\]
Since the left hand side is bounded above by
\begin{align*}
\sum_{x\in \overline X^{0}} \pr{\link_{\uparrow}(x) \text{ or }\link_{\downarrow}(x) \text{ fails}}
\end{align*}
it suffices to prove that for each $x\in \overline X^{0}$
\begin{align}\tag{$*$}\label{ineq}
\pr{\link_{\uparrow}(x) \text{ or }\link_{\downarrow}(x) \text{ fails}}<\frac 1 {|\overline X^{0}|}.
\end{align}
As $|\overline X^{0}|$ is bounded by a polynomial in $r$, but by Corollary~\ref{cor:probability} the probability on the left decreases exponentially in $r$, hence the inequality~\eqref{ineq} holds for all $r$ greater than some $R(m)$.

After finding an orientation on $\mathcal W$ such that $\alink$ and $\dlink$ are nonempty and connected, we consider the lawful subcomplex $Y\subset \overline X$ and the map $\overline X \xrightarrow{\phi} S^1$ induced by the orientation whose restriction to $Y$ lifts to a Morse function $\widetilde Y\to \reals$. By Theorem~\ref{bestvina-brady thm} the group $\ker(\pi_1Y\to \integers)$ is finitely generated. Consequently, its quotient $N=\ker(\pi_1\overline X\to \integers)$ is also finitely generated.
 To see that $\pi_1\overline X\to \integers$ is nontrivial, observe that $X^1$ has a positively directed closed path since $\overline X$ is compact and each $\alink$ is nonempty.
 \end{proof}

\section{Local quasiconvexity and Coxeter groups with nonpositive sectional curvature}
\subsection{Negative sectional curvature and local quasiconvexity}
\begin{defn}[Sectional curvature]
An \emph{angled $2$-complex} is a $2$-complex $Y$
with an \emph{angle} $\measuredangle (e)\in \mathbb R$
assigned to each edge $e$ of $\link(y)$ for each $y\in Y^{0}$. As edges in $\link(y)$ correspond to corners of $2$-cells at $y$, we regard the angles as assigned to corners of $2$-cells at $y$. 
The curvature at a $2$-cell $f$ of $Y$ is given by
\[\kappa(f) = 2\pi - \sum_{e\in \text{Corners}(f)} \deficiency(e)\]
where $\deficiency(e) = \pi - \measuredangle(e)$.
The \emph{curvature} of $Y$ at $y$ is given by
\begin{equation}\label{curvatureformula}
\kappa(y) = 2\pi - \pi\euler(\link(y))  +\sum_{e\in\text{Corners$(y)$ }}\measuredangle (e) = (2-v)\pi +\sum \deficiency(e).
\end{equation}

A \emph{section} of a combinatorial $2$-complex $Y$
at the $0$-cell $y$ is a combinatorial immersion $(S,s)\to (Y,y)$.
A section is \emph{regular} if $\link(s)$ is finite, connected, nonempty, with no valence $\leq 1$ vertex.
Pulling back the angles at a corner at  $y$ to corners at $s$, the \emph{curvature} of a section $(S,s)\to (Y,y)$ is defined to be $\kappa(s)$.
We say that $Y$ has \emph{sectional curvature} $\leq\alpha$ at $y$
if all regular sections of $Y$ at $y$ have curvature $\leq\alpha$.
Finally, $Y$ has \emph{sectional curvature} $\leq \alpha$ if each $2$-cell has curvature $\leq \alpha$ and $Y$ has sectional curvature $\leq \alpha$ at each $0$-cell.\end{defn}

\begin{defn}[Quasiconvexity]
Let $G$ be a group with a finite generating set $S$ and the Cayley graph $\Gamma(G,S)$. A subgroup $H$ of $G$ is \emph{quasiconvex}
if there is a constant $L\geq 0$ such that
every geodesic in $\Gamma(G,S)$
between two elements of $H$ lies in the $L$-neighborhood of $H$. When $G$ is hyperbolic, the quasiconvexity of $H$ is independent of the generating set of $G$ \cite{Short91}.
A group $G$ is \emph{locally quasiconvex}
if every finitely generated subgroup of $G$ is quasiconvex.
Every quasiconvex subgroup of a hyperbolic group
is finitely presented \cite{Short91}. Thus a locally quasiconvex
hyperbolic group is coherent.
\end{defn}

The main result about negative sectional curvature is as follows \cite{WiseSectional02, MartinezPedrozaWiseSectional}:

\begin{thm}\label{thm:negative sectional}
If $Y$ is a compact, piecewise Euclidean nonpositively curved $2$-complex whose associated angles have negative sectional curvature, then $\pi_1Y$ is locally quasiconvex.
\end{thm}

The following is known about locally quasiconvex Coxeter groups:
\begin{prop}\label{prop:locally quasiconvex}For each $r\geq 3$ there exists $N(r)$
such that for all $m>N(r)$ the group $G_{(r,m)}$
is locally quasiconvex.\end{prop}
We briefly review two ways of proving Proposition~\ref{prop:locally quasiconvex}. One method to prove Proposition \ref{prop:locally quasiconvex} is from \cite{McCammondWiseCoherence} or \cite[Thm IV]{Schupp03} and shows that a Coxeter group $G_{(r,m)}$ is locally quasiconvex whenever $m\geq \frac 3 2 r$.
\begin{com}check
\end{com}
We shall focus on reviewing conditions ensuring negative sectional curvature so that Theorem~\ref{thm:negative sectional} provides Proposition~\ref{prop:locally quasiconvex}.

As in Section~\ref{sec:euler}, let $X$ be the standard $2$-complex of the presentation of $G=G_{(r,m)}$ and let $\overline X$ be the compression of a finite cover of $X$ corresponding to a finite index torsion-free subgroup of $G$. If each 3-generator Coxeter subgroup of $G$ is infinite (i.e. $\frac{1}{m_{ij}}+\frac{1}{m_{jk}} +\frac{1}{m_{ki}} \leq 1$) then there is a natural metric of nonpositive curvature on $\overline X$ induced by metrizing each $2$-cell as a regular Euclidean polygon. The previous condition is equivalent to the nonpositive curvature of all sections $(S,s)\to (\overline X,x)$ where $S$ is a disc. Thus we say that $G$ has \emph{nonpositive planar sectional curvature}, when all 3-generator Coxeter subgroups are infinite.
Finally, if all exponents satisfy $m_{ij}>\frac{r(r-1)}{2(r-2)}$ then $\overline X$ has negative sectional curvature \cite[Thm~13.3]{WiseSectional02}.

\subsection{Nonpositive sectional curvature}
Let $X$ denote the standard $2$-complex of the presentation of Coxeter group $G$ and let $\overline X$ denote the compression of a cover of $X$ corresponding to a finite index torsion-free subgroup.
There is a surprisingly elegant characterization of nonpositive sectional curvature of $\overline X$ in terms of the Euler characteristic of Coxeter subgroups of $G$.

\begin{thm} 
The following are equivalent:
\begin{enumerate}
\item $\overline X$ has nonpositive sectional curvature
\item $\euler(H)\leq 0$ for each nontrivial Coxeter subgroup $H\subset G$ whose associated graph $\Upsilon_H$ is connected but not a tree.
\end{enumerate}
\end{thm}
\begin{proof}$(1)\Rightarrow (2)$: Suppose $\euler(H)> 0$ and $\Upsilon_H$ is connected and not a tree. 
We can assume that $\Upsilon_H$ has no valence $1$ vertex, since the Coxeter subgroup $H'$ associated to the subgraph $\Upsilon _{H'}$ of $\Upsilon_H$ obtained by removing a valence $1$ vertex satisfies $\euler(H')\geq \euler(H)$ by Equation~\eqref{eulerformula}.
A section at a $0$-cell of $\overline X$ whose vertices correspond to the generators of $H$ has curvature $2\pi\euler(H)$ by comparing Equations~\eqref{eulerformula} and \eqref{curvatureformula}.

$(2)\Rightarrow (1)$: 
Let $x$ be a $0$-cell of $\overline X$. It suffices to consider sections corresponding to the full subgraphs of $\link(x)$. Indeed $\deficiency(e)> 0$ for each edge $e$ since each angle is $< \pi$ and thus adding edges increases $\kappa$ by the second part of Equation~\eqref{curvatureformula}. 
Any regular section corresponding to a full subgraph is isomorphic to the associated graph $\Upsilon_H$ of a Coxeter subgroup $H$ and the curvature of the section equals $2\pi\euler(H)$. Thus if the section has positive curvature then $\euler(H)>0$.
\end{proof}

\begin{prob}\label{prob:positive incoherent}
Let $G$ have a nonpositive planar sectional curvature with $\euler(G)>0$ and $\Upsilon_G$ connected and not a tree.
Is it true that $\pi_1G$ is incoherent?
\end{prob}

We hope that the methods used here can be applied to an
appropriate finite index subgroup.
An affirmative answer to Problem~\ref{prob:positive incoherent} would be a step in proving the following:\
\begin{conj}
If $G$ has nonpositive planar sectional curvature then
the following are equivalent:
\begin{enumerate}
\item $G$ is coherent
\item$\overline X$ has nonpositive sectional curvature.
\end{enumerate}
\end{conj}

\bibliographystyle{alpha}
\bibliography{wise}

%
%
\end{document}